\documentclass[review,hidelinks,onefignum,onetabnum]{siamart220329}
\usepackage{booktabs}
\usepackage{hyperref}
\usepackage{xcolor}

\usepackage{lipsum}
\usepackage{amsfonts}
\usepackage{graphicx}
\usepackage{epstopdf}
\usepackage{algorithmic}
\usepackage{ragged2e}
\ifpdf
  \DeclareGraphicsExtensions{.eps,.pdf,.png,.jpg}
\else
  \DeclareGraphicsExtensions{.eps}
\fi

\newsiamremark{remark}{Remark}
\newsiamremark{hypothesis}{Hypothesis}
\crefname{hypothesis}{Hypothesis}{Hypotheses}
\newsiamthm{claim}{Claim}

\headers{Inverse $\mathcal{Z}$-Transform of Rational Functions}{M. Vaez, A. Hosseini, and K. Jamshidi}

\title{An Alternative Approach to\\
	Inverse $\mathcal{Z}$-Transform of Rational Functions 
}

\author{MohammadJavad Vaez\thanks{\parbox[t]{0.8\textwidth}{\raggedright A master student in the School of Mathematics, Statistics, and Computer Science, College of Science, University of Tehran, Tehran, Iran (\email{mohammadjavadvaez@gmail.com}).}}
	\and Alireza Hosseini\thanks{School of Mathematics, Statistics, and Computer Science, College of Science, University of Tehran, Tehran, Iran (\email{hosseini.alireza@ut.ac.ir}).}\and Kamal Jamshidi\thanks{Faculty of Computer Engineering, University of Isfahan, Isfahan, Iran (\email{jamshidi@eng.ui.ac.ir}).}}

\usepackage{amsopn}

\ifpdf
\hypersetup{
  pdftitle={An Alternative Approach to Inverse Z-Transform of Rational Functions},
  pdfauthor={M. Vaez, A. Hosseini, and K. Jamshidi}
}
\fi

\hypersetup{
    colorlinks=true,
    linkcolor=blue,     
    urlcolor=blue,      
    citecolor=blue,     
}

\begin{document}

\maketitle

\begin{abstract}
This paper introduces a novel method for calculating the inverse $\mathcal{Z}$-transform of rational functions. Unlike some existing approaches that rely on partial fraction expansion and involve dividing by $z$, the proposed method allows for the direct computation of the inverse $\mathcal{Z}$-transform without such division. Furthermore, this method expands the rational functions over real numbers instead of complex numbers. Hence, it doesn't need algebraic manipulations to obtain a real-valued answer. Furthermore, it aligns our method more closely with established techniques used in integral, Laplace, and Fourier transforms. In addition, it can lead to fewer calculations in some cases.
\end{abstract}

\begin{keywords}
inverse $\mathcal{Z}$-transform, partial fraction expansion, discrete-time signals, discrete-time systems, Cauchy residue calculus, teaching methodology
\end{keywords}

\begin{MSCcodes}
93-08, 93C55, 93C62, 94A12
\end{MSCcodes}

\section{Introduction}
The inverse $\mathcal{Z}$-transform plays a fundamental role in the field of digital signal processing and system analysis. It allows us to recover the original time-domain representation of a discrete-time signal from its corresponding frequency-domain representation in the $z$-domain. This transformation is invaluable in various areas, including digital filter design, control systems, and communication systems, as well as in the analysis and solution of difference equations (recurrence relations).

Traditionally, the computation of the inverse $\mathcal{Z}$-transform of rational functions has been tackled using established techniques, such as long division, contour integration, and partial fraction expansion \cite{moreira-2012,Juric-2023}.

One commonly adopted approach, proposed by \cite{moreira-2012}, employs a comprehensive method based on partial fraction expansion. However, unlike conventional methods for integration, Laplace transform, and Fourier transform, this approach requires an initial division of the assumed function by $z$. Interestingly, a similar division was previously proposed in an earlier paper by \cite{HassulShahian1991}, leading to discussions on the seeming differences in results. Recently, an alternative method was proposed in \cite{Juric-2023}, which eliminates the need for partial fraction expansion but still requires division by $z$. Additionally, if the assumed function has complex poles, further algebraic manipulations are required to obtain a real-valued result.

On the other hand, in a prior work by \cite{chen-2004}, the authors mentioned that dividing by $z$ increases the burden of calculation for some cases. Moreover, they attempted to align partial fraction expansion with the Laplace transform by avoiding such division and finding some new formulas for inverse $\mathcal{Z}$-transform of rational functions. Nonetheless, they did not address cases involving complex poles of degree higher than one.

In this paper, we aim to generalize the overlooked method proposed by \cite{chen-2004} to accommodate arbitrary degrees of complex poles. By doing so, we introduce a novel and versatile method for computing the inverse $\mathcal{Z}$-transform of rational functions.

In the forthcoming sections of this paper, we will briefly review previous methods, followed by a detailed presentation of our proposed method. We will highlight its underlying theorems and demonstrate its applicability through illustrative examples. Additionally, we will compare the performance of our method with existing approaches.

\section{computation of inverse $\mathcal{Z}$-transform}
In this paper, we primarily focus on the unilateral $\mathcal{Z}$-transform, which is particularly important in digital signal processing, as it deals with signals in the positive time domain. As mentioned in prior research \cite{moreira-2012}, there are at least three well-known methods for computing the inverse $\mathcal{Z}$-transform of rational functions. Recently, a new method has been proposed by\cite{Juric-2023}. In the sequel, we introduce the main methods proposed previously in this literature.
\subsection{Contour Integration}
The more basic method, applicable to any function, involves contour integration. If we denote the $\mathcal{Z}$-transform of the discrete-time signal $x[n]$ as $X(z)$, the inverse $\mathcal{Z}$-transform can be expressed as
\begin{equation} \label{contour_int}
	x[n]=\frac{1}{2\pi i}\oint_{C}\, X(z)z^{n-1}\, dz.
\end{equation}
where $C$ is a counterclockwise contour that encircles the origin and lies entirely in the region of convergence (ROC). For unilateral $\mathcal{Z}$-transform, the interior of $C$ must contain all poles of $X(z)$. Assuming that $X(z)$ has $K$ distinct poles $z_1, z_2, \ldots, z_K$, we can evaluate \eqref{contour_int} by applying Cauchy's residue theorem
\begin{equation} \label{sum_of_res}
	x[n]=\sum_{k=1}^{K}\displaystyle\underset{z=z_k}{\mathrm{Res}} \left(X(z)z^{n-1}\right).
\end{equation}
If $z_i$ is a pole of multiplicity $m$, then
\begin{equation} \label{calculating_residues}
	\displaystyle\underset{z=z_k}{\mathrm{Res}} \left(X(z)z^{n-1}\right)=\frac{1}{(m-1)!}\lim_{z\to z_i}\frac{d^{m-1}}{{dz}^{m-1}}\left[(z-z_i)^mX(z)z^{n-1}\right].
\end{equation}
\newline \par
The other two methods mentioned in \cite{moreira-2012} specifically trigger the rational functions \footnote{Of course, partial fraction expansion can also be applied to non-rational functions involving infinite series and the Mittag-Leffler theorem \cite{chaudhary-2011}. Here, however, we focus solely on rational functions.}. A rational function is nothing but the quotient of two polynomials, say
\begin{equation}
	X(z)=\frac{b_0z^p+b_1z^{p-1}+b_2z^{p-2}+\ldots+b_p}{z^q+a_1z^{q-1}+a_2z^{q-2}+\ldots+a_q}.
\end{equation}
\subsection{long division}
The next approach is the long division method. This \-method generally does not yield a closed-form expression for $x[n]$. Thus, for the sake of brevity, interested readers can refer to Example 2.12 in \cite{phillips2014}. Nonetheless, it is important to mention that, according to \cite{moreira-2012}, using this method allows us to easily deduce the following relationships:

\begin{equation} \label{first terms}
	\text{If } q>p, \text{ then } x[0]=x[1]=\ldots=x[q-p-1]=0;
\end{equation}
and
\begin{equation} \label{long_div_2}
	\text{If } q\geq p, \text{ then } x[q-p]=b_0.
\end{equation}

\subsection{Partial Fraction Expansion}
The primary method of focus in both \cite{moreira-2012} and the present study is partial fraction expansion, also known as partial fraction decomposition. However, there are two notable distinctions between the approach employed in this study and the one proposed in \cite{moreira-2012}:
\begin{itemize}
	\item The approach of this study decomposes polynomials over the field of real numbers while their approach decompose them over complex numbers.
	
	\item In this study, the expansion is applied to $X(z)$, while the approach in \cite{moreira-2012} expands $\frac{X(z)}{z}$.
\end{itemize}
In fact, these differences align the method employed in this study with the conventional techniques commonly used for integral, Laplace transform, and Fourier transform.

Table \ref{tab:inverse_z_transforms}, originally presented in \cite{moreira-2012}, is utilized as a reference for comparing the inverse $\mathcal{Z}$-transform of a rational function with our proposed method.

\begin{table}
	\caption{Inverse $\mathcal{Z}$-Transforms for the General Terms in Partial-Fraction Expansion, Taken from \cite{moreira-2012}}
	\label{tab:inverse_z_transforms}
	\begin{tabular}{@{}lll@{}}
		\toprule
		\textbf{Type of poles}                  & \textbf{Term}                                                                                                  & \textbf{Inverse $\mathcal{Z}$-transform}                                 \\ \midrule
		Single/multiple poles at $z=0$ & $\frac{A}{z^{n_0}}$                                                                                   & $A\delta(n-n_0)$                                                \\ \midrule
		Single real pole               & $\frac{Az}{z-r}$                                                                                      & $Ar^n,\text{ } n\geq0$                                                 \\ \midrule
		Multiple real pole             & $\frac{Az}{(z-r)^q}$                                                                                  & $\frac{Ar^{n-q+1}}{(q-1)!}\prod_{i=0}^{q-2}(n-i),\text{ } n\geq0$      \\ \midrule
		Single complex poles         & $\begin{aligned}
			&\begin{tabular}{@{}l@{}}
				$\frac{Ae^{i\phi}z}{z-re^{i\theta}}+\frac{Ae^{-i\phi}z}{z-re^{-i\theta}},$\\
				$A,r\in\mathbb{R}_+$
			\end{tabular}
		\end{aligned}$ & $\begin{aligned}
			&\begin{tabular}{@{}l@{}}
				$2Ar^n\cos(n\theta+\phi),\text{ } n\geq0$
			\end{tabular}
		\end{aligned}$ \\ \midrule
		Multiple complex poles         & $\begin{aligned}
			&\begin{tabular}{@{}l@{}}
				$\frac{Ae^{i\phi}z}{(z-re^{i\theta})^q}+\frac{Ae^{-i\phi}z}{(z-re^{-i\theta})^q},$\\
				$A,r\in\mathbb{R}_+$
			\end{tabular}
		\end{aligned}$ & $\begin{aligned}
			&\begin{tabular}{@{}l@{}}
				$\frac{2Ar^{n-q+1}}{(q-1)!}\cos\left((n-q+1)\theta+\phi\right)$ \\
				$\times \prod_{i=0}^{q-2}(n-i),\text{ } n\geq0$
			\end{tabular}
		\end{aligned}$ \\ \bottomrule
	\end{tabular}
\end{table}

\subsection{The Method Proposed by \cite{Juric-2023}}
The primary method proposed by \cite{Juric-2023} for the unilateral $\mathcal{Z}$-transform is outlined as follows:

Let $\frac{X(z)}{z}=\frac{N(z)}{D(z)}$ and $D(z)$ has distinct roots $z_1,z_2,\ldots,z_K$ with multiplicity $m_1,m_2,\ldots,m_K$ respectively. Then, the expression for $x[n]$ is given by
\begin{equation} \label{Juric1}
	x[n]=\sum_{k=1}^{K}\sum_{j=0}^{m_k-1}c_{k,m_k-1-j}\binom{n}{j}z_k^{n-j}.
\end{equation}
If $z_k=0$, $\binom{n}{j}z_k^{n-j}$ in Equation \eqref{Juric1} should be replaced with $\delta[n-j]$.The coefficients $c_{k,j}$, where $k=1,2,\ldots,K$ and $j=0,1,\ldots,m_k-1$, are calculated using the following formula:
\begin{equation} \label{Juric2}
	c_{k,j} = \frac{1}{j!D_k(z_k)} \left(N^{(j)}(z_k) - \sum_{l=0}^{j-1}c_{k,l}(j)_lD_k^{(j-l)}(z_k)\right).
\end{equation}
Here, $(j)_l$ represents the falling factorial:
\begin{equation} \label{falling_factorial}
	\left(j\right)_l:={\overbrace{j(j-1)\dots(j-l+1)}^{\text{$l$ terms}}}.
\end{equation}
Moreover, $N^{(j)}(z), D_k^{(j)}(z)$ denote the $j$-th derivatives of $N(z), D_k(z)$ respectively, and $D_k(z)=\begin{cases}
			(z-z_k)^{m_k}D(z), & \text{if $z\neq z_k$}\\
            \lim_{z\to z_k}(z-z_k)^{m_k}D(z), & \text{if $z=z_k$}
		 \end{cases}.$ The equation  \eqref{Juric2} is recursive and complicated. For this reason, the first few terms of that has been calculated in \cite{Juric-2023}:
\begin{equation}\label{Juric3}
	c_{k,0}=\frac{N(z_k)}{D_k(z_k)},
\end{equation}
\begin{equation}\label{Juric4}
	c_{k,1}=\frac{N'(z_k)-c_{k,0}D_k'(z_k)}{D_k(z_k)},
\end{equation}
\begin{equation}\label{Juric5}
	c_{k,2}=\frac{N''(z_k)-c_{k,0}D_k''(z_k)-2c_{k,1}D_k'(z_k)}{2D_k(z_k)}.
\end{equation}
\newline \par
In the subsequent part, we will explain the proposed method.

\section{The Proposed Method}
In this approach, we utilize  the conventional partial fraction expansion of a rational function over the field of real numbers, as described in \cite{larson-2010}. Let $X(z)=\frac{N(z)}{D(z)}$. If $deg(N)\geq deg(D)$, divide the numerator by the denominator to obtain $\frac{N(z)}{D(z)}=\text{polynomial}+\frac{N_1(z)}{D(z)}$ (where $N_1(z)$ represents the remainder from the division of $N(z)$ by $D(z)$). Next, factorize the denominator into factors of the following forms: $(z-r)^u$ and $(z^2-2az+(a^2+b^2))^k$, where $r,a,b\in \mathbb{R}$ (Note that the second form of factors is irreducible over the real numbers).

For each factor of the form $(z-r)^u$, the partial fraction expansion should include a sum of $u$ fractions as follows:
\[\frac{A_1}{z-r}+\frac{A_2}{(z-r)^2}+\ldots+\frac{A_u}{(z-r)^u}.\]
For each factor of the form $(z^2-2az+(a^2+b^2))^k$, the partial fraction expansion should include a sum of $k$ fractions:
\[\frac{B_1z+C_1}{z^2-2az+(a^2+b^2)}+\frac{B_2z+C_2}{(z^2-2az+(a^2+b^2))^2}+\ldots+\frac{B_kz+C_k}{(z^2-2az+(a^2+b^2))^k}.\]
That is, the partial fraction expansion of $X(z)$ is the following:
\begin{equation}
	X(z)=\frac{N(z)}{D(z)}=P(z)+\sum_{h=1}^{v}\sum_{j=1}^{u_h}\frac{A_{hj}}{(z-r_h)^j}+\sum_{h=1}^{w}\sum_{j=1}^{k_h}\frac{B_{hj}z+C_{hj}}{(z^2-2a_hz+(a_h^2+b_h^2))^j}.
\end{equation}
There are several approaches to determine the coefficients $A_{hj},B_{hj},C_{hj}$. The most direct method involves multiplying both sides of the equation by the common denominator, $D(z)$. This yields a polynomial equation where the left-hand side simplifies to $N(z)$, while the right-hand side consists of coefficients expressed as linear combinations of the constants $A_{hj},B_{hj},C_{hj}$. By equating the coefficients of corresponding terms, we obtain a system of linear equations. This system always possesses a unique solution. Standard methods of linear algebra can be employed to find this solution. Alternatively, limits can be utilized, as demonstrated in \cite{bluman-1984}.

Therefore, by leveraging the linearity property of the inverse $\mathcal{Z}$-transform, it is sufficient to calculate the inverse $\mathcal{Z}$-transform of each term. In the following sections, we will determine these inverse transforms.
\subsection{Powers of $z$}
For the polynomial terms and terms containing poles at the origin, it is sufficient to know the inverse $\mathcal{Z}$-transform of the powers of $z$, which can be easily found.
\begin{equation}
	\mathcal{Z}^{-1}(z^k)=\delta[n+k].
\end{equation}
where $k$ can be an arbitrary integer, and $\delta[\cdot]$ represents the discrete-time unit impulse function. This equation was also addressed in \cite{phillips2014,moreira-2012}.
\subsection{Terms Containing Real Non-zero Poles}
\begin{theorem}
\begin{align}
    \label{real_poles_1}
		\mathcal{Z}^{-1}\left(\frac{1}{\left(z-a\right)^k}\right) &=\binom{n-1}{k-1}a^{n-k} \\
        &=\binom{n-1}{n-k}a^{n-k}.  \label{real_poles_2}
\end{align}
\end{theorem}
\begin{remark}  \label{remark_1}
    Note that $\binom{\nu}{\kappa}=\frac{\nu(\nu-1)\ldots(\nu-\kappa+1)}{\kappa!}$ and is zero when $\kappa>\nu$.
\end{remark}
Additionally, note that as mentioned in \cite{Juric-2023}, $\lim_{\epsilon\to 0} \binom{n-1}{k-1}\epsilon^{n-k}=\delta[n-k]$. This implies that the terms containing poles at the origin are the limiting case of the real non-zero poles.
\begin{proof}
	Since $\binom{n-1}{k-1}=\binom{n-1}{n-k}$, it suffices to prove Equation (\ref{real_poles_1}).
	According to \cite[Theorem 1]{chen-2004}, we know that
	\begin{align} \label{removal_of_u}
	    \mathcal{Z}^{-1}\left(\frac{1}{z-a}\right) &=a^{n-1}u\left[n-1\right]  \nonumber \\
       &=\binom{n-1}{1-1}a^{n-1}
	\end{align}
 
	Hence, the statement holds for $k=1$. Moreover, based on \cite[Theorem 2]{chen-2004}, for $k\geq2$
	\begin{align}
		\mathcal{Z}^{-1}\left(\frac{1}{\left(z-a\right)^k}\right) &= \frac{1}{(k-1)!}\mathcal{Z}^{-1}\left(\frac{(k-1)!}{\left(z-a\right)^k}\right) \nonumber \\
		&= \frac{(n-1)(n-2)\ldots(n-(k-1))}{(k-1)!}a^{n-k}u\left[n-1\right].
	\end{align}
	Note that for each integer $n$, we have $(n-1)(n-2)\ldots(n-(k-1))u\left[n-1\right]=(n-1)(n-2)\ldots(n-(k-1))u\left[n-k\right]$
	Therefore,
	\begin{align}
		\mathcal{Z}^{-1}\left(\frac{1}{\left(z-a\right)^k}\right) &= \frac{(n-1)(n-2)\ldots(n-(k-1))}{(k-1)!}a^{n-k}u\left[n-k\right] \nonumber\\
        &= \binom{n-1}{k-1}a^{n-k}u\left[n-k\right] \nonumber\\
        &=\binom{n-1}{k-1}a^{n-k}.
	\end{align}
\end{proof}
\subsection{Terms Containing Complex Poles}
\begin{theorem} \label{f0_theorem} Let $a\pm ib=re^{\pm i\theta}$. Suppose
\begin{equation}
		f_0[n]=\mathcal{Z}^{-1}\left(\frac{1}{\left(z^2-2az+\left(a^2+b^2\right)\right)^k}\right).
	\end{equation}
    Then
    \begin{equation} \label{main_theorem}
        f_0[n]=\frac{2~(-1)^{k-1}r^{n-2k}}{\left(2\sin{\theta}\right)^{2k-1}}\sum_{j=0}^{k-1}{\left(-1\right)^j\binom{n-1}{j}\binom{n-(k+1+j)}{k-1-j}\sin{\left(\left(n-2j-1\right)\theta\right)}}.
    \end{equation}
\end{theorem}
Since the proof of this theorem is too long and technical, we have provided it in Appendix \ref{app:appendixA}.

It is of more mathematical rigor to denote $f_0[n]$ as $f_0[n](a,b,k)$; however, in the sequel we use the simpler notation $f_0[n]$ to avoid overcomplication and to enhance readability, especially when the parameters $a$, $b$, and $k$ are understood from the context.

\begin{corollary} \label{f1_theorem} Let $a\pm ib=re^{\pm i\theta}$. Suppose
	\begin{equation}
		f_1[n]=\mathcal{Z}^{-1}\left(\frac{z}{\left(z^2-2az+\left(a^2+b^2\right)\right)^k}\right).
	\end{equation}
    Then
    \begin{equation} \label{corollary}
    f_1[n]=f_0[n+1]
    \end{equation}
    where $f_0$ is defined in \eqref{main_theorem}.
\end{corollary}
\begin{proof}
	We know that for each digital signal $x[n]$, $\mathcal{Z}(x[n+1])=zX(z)-zx[0]$. Furthermore, if
	\begin{equation}
		f_0[n]=\mathcal{Z}^{-1}\left(\frac{1}{\left(z^2-2az+\left(a^2+b^2\right)\right)^k}\right),
	\end{equation}
	then based on {\eqref{first terms}}, for each $k\geq1, f_0[0]=0$. As a result,
	\begin{equation}
		\mathcal{Z}(f_0[n+1])=zF_0(z)=\frac{z}{\left(z^2-2az+\left(a^2+b^2\right)\right)^k}.
	\end{equation}
	So, if we replace $n$ with $n+1$ in \eqref{main_theorem}, we will get the desired result.
\end{proof}
Given formulas \cref{main_theorem,corollary}, we can find the inverse $\mathcal{Z}$-transform of our building blocks of partial fraction expansion.
\begin{corollary}\label{general_cor} Let $g[n]=\mathcal{Z}^{-1}\left(\frac{A_1z+A_0}{\left(z^2-2az+\left(a^2+b^2\right)\right)^k}\right)$. Then
	\begin{align} \label{building_blocks}
	    g[n] &= A_1f_1[n] + A_0f_0[n] \nonumber \\
            &= A_1f_0[n+1] + A_0f_0[n]
	\end{align}
    where $f_0$ and $f_1$ are defined in \eqref{main_theorem} and \eqref{corollary} respectively.
    
    This statement includes both \cref{f0_theorem,f1_theorem}.
	
\end{corollary}

\begin{remark}\label{remark_2}
    Note that based on \cref{first terms,main_theorem}, the terms $f_0[n]~~\forall~n<2k$ are zero. Thus, we can always multiply $f_0[n]$ by a term $u[n-2k]$. Similarly, based on \cref{first terms,f1_theorem}, we can multiply $f_1[n]$ by a term $u[n-2k+1]$. These terms will become necessary only when simplifying the summations (otherwise, based on remark \eqref{remark_1}, the binomial coefficients, themselves, satisfy this condition).
\end{remark}

\theoremstyle{definition}
\newtheorem{example}{Example}[section]
\section{Examples}
Now, we give some examples to illustrate how to use  \eqref{building_blocks}.
\begin{example}
	In this example, we will compute the inverse $\mathcal{Z}$-transform of $\frac{1}{z^2+1}$. Hence, we can use the equation \eqref{main_theorem} with $k=1$, and $a\pm ib=e^{\pm i\frac{\pi}{2}}$; i.e. $r=1, \theta=\frac{\pi}{2}$. Therefore, the summation used in \eqref{main_theorem} will have just one summand for $j=0$. Then, both the binomial coefficients will be equal to 1 (except for $f_0[0], f_0[1]$, based on Remark \eqref{remark_2}). In summary, we have
	\begin{equation} \label{proposed1}
		\mathcal{Z}^{-1}\left(\frac{1}{z^2+1}\right)=\frac{u[n-2]}{\sin\frac{\pi}{2}}\sin\left((n-1)\frac{\pi}{2}\right)=-u[n-2]\cos\left(\frac{n\pi}{2}\right).
	\end{equation}
	
	Now, let's find the answer using the method described in \cite{moreira-2012}. According to that method, first, we must decompose $Y(z)=\frac{X(z)}{z}$ over the field of complex numbers; that is
	\begin{equation} \label{moreira11}
		\frac{1}{z(z^2+1)}=\frac{A}{z}+\frac{B}{z-i}+\frac{B^*}{z+i}\text{ where } B=B_1e^{i\phi},
	\end{equation}
	and the coefficients are obtained via these limits:
	\[A=\lim_{z\to 0}zY(z)=\lim_{z\to 0}\frac{1}{z^2+1}=1\]
	and
	\[B=\lim_{z\to i}(z-i)Y(z)=\lim_{z\to i}\frac{1}{z(z+i)}=\frac{-1}{2}.\]
	By multiplying both sides of \eqref{moreira11} by $z$, we get
	\[\frac{1}{z^2+1}=1+\frac{-z/2}{z-i}+\frac{-z/2}{z+i}.\]
	So, using the Table \ref{tab:inverse_z_transforms}, we obtain
	\begin{equation} \label{moreira12}
		x[n]=\delta[n]+2\left(\frac{-1}{2}\right)\cos\left(\frac{\pi}{2}n\right)=\delta[n]-\cos\left(\frac{n\pi}{2}\right).
	\end{equation}
	It's easy to see that the results of \cref{proposed1,moreira12} are identical.
	
	Also, let's solve it by the method proposed by \cite{Juric-2023} mentioned in \cref{Juric1,Juric2}.
	Since $\frac{X(z)}{z}=\frac{1}{z(z^2+1)}$, we have $z_1=0, z_2=i, z_3=-i, m_1=m_2=m_3=1, P(z)=1, Q(z)=z(z^2+1)$. Therefore,
	\begin{equation}
		x[n]=c_{1,0}\delta[n]+\sum_{k=2}^{3}c_{k,0}z_k^n \text{ where } c_{k,0}=\frac{P(z_k)}{Q_k(z_k)}.
	\end{equation}
	This gives us
	\begin{align}
		x[n] &= c_{1,0}\delta[n] + c_{2,0}z_2^n + c_{3,0}z_3^n \nonumber \\
		&= \left[\frac{1}{z^2+1}\right]_{z=0}\delta[n] + \left[\frac{1}{z(z+i)}\right]_{z=i}i^n + \left[\frac{1}{z(z-i)}\right]_{z=i}(-i)^n \nonumber \\
		&= \delta[n]-\frac{i^n}{2}-\frac{(-i)^n}{2}=\delta[n]-\frac{(i^n+(-i)^n)}{2}=\delta[n]-\mathfrak{Re}(i^n) \nonumber \\
		&= \delta[n]-\cos(\frac{n\pi}{2}).
	\end{align}
\end{example}
\begin{example}
	Let's consider a more complex example. Our goal is to find the inverse $\mathcal{Z}$-transform of the expression $X(z)=\frac{Az+B}{(z^2-2az+(a^2+b^2))^3}$. To solve it using our proposed method, we can refer to Corollary \eqref{general_cor}. Using these equations, we obtain:
	\[x[n]=Af_0[n]+Bf_0[n+1],\]
	where
    \begin{equation}
        f_0[n]=\frac{2~(-1)^{k-1}r^{n-2k}}{\left(2\sin{\theta}\right)^{2k-1}}\sum_{j=0}^{k-1}{\left(-1\right)^j\binom{n-1}{j}\binom{n-(k+1+j)}{k-1-j}\sin{\left(\left(n-2j-1\right)\theta\right)}}.
    \end{equation}
    in which $r=\sqrt{a^2+b^2}$, and $\theta=Arg(a+ib)$.
	As a result,
	\[x[n]=A\delta[n-5]+\frac{r^{n-6}u\left[n-6\right]}{16\left(\sin{\theta}\right)^5}\left(S_0+S_1+S_2\right),\]
	where $S_0$, $S_1$, and $S_2$ are defined as follows:
	\[S_0=Ar\binom{n-3}{2}\sin((n-1)\theta)+B\binom{n-4}{2}\sin(n\theta),\]
	\[S_1=Ar\binom{n}{1}\binom{n-4}{1}\sin((n-3)\theta)+B\binom{n-1}{1}\binom{n-5}{1}\sin((n-2)\theta),\]
    and
	\[S_2=Ar\binom{n}{2}\sin((n-5)\theta)+B\binom{n-1}{2}\sin((n-5)\theta).\]
	\newline
	If we want to obtain the result using the method proposed by \cite{moreira-2012}, we need to perform the partial fraction decomposition of $Y(z)=\frac{X(x)}{z}$. The decomposition is as follows:
	\begin{align*}
		\frac{Az+B}{z(z^2-2az+(a^2+b^2))^3} = \frac{A_{1,1}}{z} &+ \frac{A_{2,1}}{z-(a+ib)} + \frac{A_{2,1}^*}{z-(a-ib)} \\
		&+ \frac{A_{2,2}}{(z-(a+ib))^2} + \frac{A_{2,2}^*}{(z-(a-ib))^2} \\
		&+ \frac{A_{2,3}}{(z-(a+ib))^3} + \frac{A_{2,3}^*}{(z-(a-ib))^3},
	\end{align*}
	where
	\[A_{1,1}=\lim_{z\to 0}zY(z)=\lim_{z\to 0}\frac{Az+B}{(z^2-2az+(a^2+b^2))^3}=\frac{B}{(a^2+b^2)^3}=\frac{B}{r^6}\text{ ,}\]
	\[A_{2,p}=\frac{1}{(3-p)!}\lim_{z\to a+ib}\frac{d^{3-p}}{dz^{3-p}}\left((z-(a+ib))^3Y(z)\right) \text{ for } p=1,2,3.\]
	It's easy to see that the remaining calculations using this method would be lengthy and cumbersome.
	
	Furthermore, if we want to obtain the result using the method proposed by \cite{Juric-2023} (refer to \cref{Juric1,Juric2}), the solution can be expressed as follows:
	First, let $z_1=0, z_2=a+ib, z_3=a-ib$. Then
	\[x[n]=c_{1,0}\delta[n]+\sum_{k=2}^{3}\sum_{j=0}^{2}c_{k,2-j}\binom{n}{j}z_k^{n-j}.\]
	Since $z_3=z_2^*$ and $(z_k^{n-j})^*=(z_k^*)^{n-j}$, we can write
	\[x[n]=c_{1,0}\delta[n]+2\mathfrak{Re}\left(\sum_{j=0}^{2}c_{2,2-j}\binom{n}{j}z_2^{n-j}\right).\]
	For the sake of brevity, we leave the rest of the computations to interested readers! Note that due to the parametric nature of this example and the presence of complex roots, computation of the coefficients and simplifying the results to get a neat real-valued function will be cumbersome and time-consuming.
\end{example}

\section{Supplementary Material}
\raggedright
To ensure the correctness of the proposed method and to facilitate a better comparison with the other two methods discussed in this paper, we have provided an implementation of all three methods, which can be found \href{https://github.com/mjvaez/inverse-Z-transform/blob/main/inverse_Z_transform.ipynb}{\textcolor{blue}{here}}.
\par

\section{Conclusions} \label{sec:conclusions}
In this paper, we presented a novel technique for computing the inverse $\mathcal{Z}$-transform of rational functions. Through a couple of examples, we demonstrated that compared to existing methods, our approach can result in fewer calculations in some cases. Unlike the method proposed by \cite{moreira-2012}, our approach eliminates the need for dividing by $z$ and instead utilizes the conventional partial fraction expansion over real numbers. These two characteristics make our method resemble the approaches used for computing well-known integral transforms.
Furthermore, compared to the work done by \cite{Juric-2023}, our method offers the advantage of not requiring additional algebraic manipulations to obtain a real-valued solution.

\appendix
\section{Proof of the Equation {\eqref{main_theorem}}} \label{app:appendixA}
In this section, we prove equation \eqref{main_theorem}. According to the lemma below, it can be easily seen that the first few terms of the time-domain sequence are zero.
\begin{lemma} \label{lemma}
	Let $x[n]$ be the inverse $\mathcal{Z}$-transform of
	$X(z)=\frac{1}{\left(z^2-2az+\left(a^2+b^2\right)\right)^k}$. Then for every $n$ such that $n<2k$, $x[n]=0$.\newline
\end{lemma}
	\begin{proof}
		It's an immediate result of \eqref{first terms}. Nevertheless, an alternative proof is provided in Appendix \ref{app:appendixB}.
	\end{proof}
Now assume $n\geq2k$. For these $n$'s, we are going to find the inverse $\mathcal{Z}$-transform of $X(z)=\frac{1}{\left(z^2-2az+\left(a^2+b^2\right)\right)^k}$ using another method, that is the residue method. Note that 
\[x[n]=\mathcal{Z}^{-1}\left(\frac{1}{\left(z-\left(a+bi\right)\right)^k}\times\frac{1}{\left(z-\left(a-bi\right)\right)^k}\right).\]
According to \eqref{sum_of_res}, \begin{equation} \label{x[n]_sum}
    x[n]=\displaystyle\underset{z=a+ib}{\mathrm{Res}} X(z)+\displaystyle\underset{z=a-ib}{\mathrm{Res}} X(z).
\end{equation}
Given that both poles of $X(z)$ are of degree $k$, we apply the formula outlined in \eqref{calculating_residues} to obtain
\begin{equation} \label{first_res_eq}
	\displaystyle\underset{z=a+ib}{\mathrm{Res}} X(z) = \frac{1}{\left(k-1\right)!}\left.\left[\frac{d^{k-1}}{dz^{k-1}}\left(\frac{z^{n-1}}{\left(z-\left(a-ib\right)\right)^k}\right)\right]\right\rvert_{z=a+ib},
\end{equation}
and similarly, for the conjugate pole
\begin{equation} \label{second_res_eq}
	\displaystyle\underset{z=a-ib}{\mathrm{Res}} X(z) = \frac{1}{\left(k-1\right)!}\left.\left[\frac{d^{k-1}}{dz^{k-1}}\left(\frac{z^{n-1}}{\left(z-\left(a+ib\right)\right)^k}\right)\right]\right\rvert_{z=a-ib}.
\end{equation}
We proceed with computing a general expression for terms found in \cref{first_res_eq,second_res_eq}. Considering 
$\xi$ and $\eta$ as complex conjugates, we define
\begin{equation} \label{D}
	D=\left.\left[\frac{d^{k-1}}{dz^{k-1}}\left(\frac{z^{n-1}}{\left(z-\eta\right)^k}\right)\right]\right|_{z=\xi}.
\end{equation}
Applying the general Leibniz rule yields
\begin{equation}
	D=\sum_{t=0}^{k-1}{\left(\begin{matrix}k-1\\t\\\end{matrix}\right)\left.\left[\frac{d^t}{dz^t}\left(z^{n-1}\right)\right]\right|_{z=\xi}}\left.\left[\frac{d^{k-1-t}}{dz^{k-1-t}}\left(\left(z-\eta\right)^{-k}\right)\right]\right|_{z=\xi}.
\end{equation}
Note that
\begin{equation}
	\left.\left[\frac{d^t}{dz^t}\left(z^{n-1}\right)\right]\right|_{z=\xi}=u\left[n-1-t\right]\left(n-1\right)\ldots\left(n-t\right)\xi^{n-t-1}.
\end{equation}
Since we have supposed that $n\geq2k$, given the upper bound of the summation, $t\leq k-1$, we can infer that $n-t-1 \geq k \geq 1$. This leads us to conclude that $u[n-t-1]=1$. We can then express $D$ as:
\begin{equation} \label{eq:eq2}
	\begin{aligned}
		D &= \sum_{t=0}^{k-1} \binom{k-1}{t} (n-1) \ldots (n-t) \xi^{n-t-1} \\
		&\quad \times (-1)^{k-1-t} k(k+1) \ldots (2k-2-t) (\xi-\eta)^{-(2k-1-t)}.
	\end{aligned}
\end{equation}
We also observe that:
\[\binom{k-1}{t}(n-1)\ldots(n-t)\times k(k+1)\ldots(2k-2-t)\]
\vspace{-0.5cm}
\begin{align} \label{eq:eq3}
	&= \frac{(k-1)!}{t!(k-1-t)!}\times\frac{(n-1)!}{(n-t-1)!}\times\frac{(2k-2-t)!}{(k-1)!} \nonumber \\
	&= \frac{(n-1)!}{t!(n-t-1)!}\times\frac{(2k-2-t)!}{(k-1-t)!} \nonumber \\
	&= \binom{n-1}{t}\times\frac{(2k-2-t)!}{(k-1-t)!}.
\end{align}

Substituting equation \ref{eq:eq3} into \ref{eq:eq2}, we get
\[D=\xi^{n-1}\left(-1\right)^{k-1}\left(\xi-\eta\right)^{-\left(2k-1\right)}\sum_{t=0}^{k-1}{\left(\begin{matrix}n-1\\t\\\end{matrix}\right)\frac{\left(2k-2-t\right)!}{\left(k-1-t\right)!}\xi^{-t}\times\left(-1\right)^{-t}\left(\xi-\eta\right)^t}\]
\begin{equation}
	=\xi^{n-1}\left(-1\right)^{k-1}\left(\xi-\eta\right)^{1-2k}\sum_{t=0}^{k-1}{\left(\begin{matrix}n-1\\t\\\end{matrix}\right)\frac{\left(2k-2-t\right)!}{\left(k-1-t\right)!}}\left(\frac{\eta-\xi}{\xi}\right)^t.
\end{equation}
Given that $\xi$ and $\eta$ are complex conjugates, we can express them as $\xi=re^{i\theta}$ and $\eta=re^{-i\theta}$. After substituting these values, we obtain
\[D=\left.\left[\frac{d^{k-1}}{dz^{k-1}}\left(\frac{z^{n-1}}{\left(z-re^{-i\theta}\right)^k}\right)\right]\right|_{z=re^{i\theta}}\]
\[=r^{n-1}e^{i\left(n-1\right)\theta}\left(-1\right)^{k-1}\left(2ri\ \sin{\theta}\right)^{1-2k}\sum_{t=0}^{k-1}{\left(\begin{matrix}n-1\\t\\\end{matrix}\right)\frac{\left(2k-2-t\right)!}{\left(k-1-t\right)!}\left(e^{-2i\theta}-1\right)^t}\]
\begin{equation} \label{final_D}
	=r^{n-2k}e^{i\left(n-1\right)\theta}\left(-1\right)^{k-1}\left(2i\ \sin{\theta}\right)^{1-2k}\sum_{t=0}^{k-1}{\left(\begin{matrix}n-1\\t\\\end{matrix}\right)\frac{\left(2k-2-t\right)!}{\left(k-1-t\right)!}\left(e^{-2i\theta}-1\right)^t}.
\end{equation}
Likewise, \ref{second_res_eq} is given below, where $D^*$ represents the complex conjugate of $D$.
\[D^*=\left.\left[\frac{d^{k-1}}{dz^{k-1}}\left(\frac{z^{n-1}}{\left(z-re^{i\theta}\right)^k}\right)\right]\right|_{z=re^{-i\theta}}\]
\begin{equation} \label{Dstar}
	=r^{n-2k}e^{i\left(1-n\right)\theta}\left(-1\right)^{k-1}\left(-2i\ \sin{\theta}\right)^{1-2k}\sum_{t=0}^{k-1}{\left(\begin{matrix}n-1\\t\\\end{matrix}\right)\frac{\left(2k-2-t\right)!}{\left(k-1-t\right)!}\left(e^{2i\theta}-1\right)^t}.
\end{equation}
By utilizing the equations \cref{x[n]_sum,first_res_eq,second_res_eq,D,final_D,Dstar}, we can deduce that
\begin{equation}
	f\left[n\right]=\frac{\left(-1\right)^{k-1}}{(k-1)!}~r^{n-2k}\left(2i\ \sin{\theta}\right)^{1-2k}\left(A-A^*\right),
\end{equation}
where $A$ and $A^*$ are defined as
\begin{equation}
	A=e^{i\left(n-1\right)\theta}\sum_{t=0}^{k-1}{\left(\begin{matrix}n-1\\t\\\end{matrix}\right)\frac{\left(2k-2-t\right)!}{\left(k-1-t\right)!}\left(e^{-2i\theta}-1\right)^t}
\end{equation}
and
\begin{equation}
	A^*=e^{i\left(1-n\right)\theta}\sum_{t=0}^{k-1}{\left(\begin{matrix}n-1\\t\\\end{matrix}\right)\frac{\left(2k-2-t\right)!}{\left(k-1-t\right)!}\left(e^{2i\theta}-1\right)^t}.
\end{equation}
Since $A-A^*=2i\mathfrak{Im}(A)$, we can express $x[n]$ as
\begin{equation} \label{f_n_short}
	f\left[n\right]=\frac{\left(-1\right)^{k-1}}{(k-1)!}~r^{n-2k}\left(2i\ \sin{\theta}\right)^{1-2k}\left(2i\mathfrak{Im}(A)\right).
\end{equation}
Now, based on binomial expansion of $\left(e^{2i\theta}-1\right)^t$, $A$ can be expanded as
\[A=e^{i\left(n-1\right)\theta}\sum_{t=0}^{k-1}{\left(\begin{matrix}n-1\\t\\\end{matrix}\right)\frac{\left(2k-2-t\right)!}{\left(k-1-t\right)!}\sum_{j=0}^{t}{\left(\begin{matrix}t\\j\\\end{matrix}\right)\left(e^{-2i\theta}\right)^j\left(-1\right)^{t-j}}}\]
\begin{equation}
	=e^{i\left(n-1\right)\theta}\sum_{t=0}^{k-1}\sum_{j=0}^{t}{\left(\begin{matrix}n-1\\t\\\end{matrix}\right)\frac{\left(2k-2-t\right)!}{\left(k-1-t\right)!}\left(\begin{matrix}t\\j\\\end{matrix}\right)\left(e^{-2i\theta}\right)^j\left(-1\right)^{t-j}}.
\end{equation}
By interchanging the order of summations, we can express $A$ as
\begin{equation}
	A=e^{i\left(n-1\right)\theta}\sum_{j=0}^{k-1}\sum_{t=j}^{k-1}{\left(\begin{matrix}n-1\\t\\\end{matrix}\right)\frac{\left(2k-2-t\right)!}{\left(k-1-t\right)!}\left(\begin{matrix}t\\j\\\end{matrix}\right)\left(e^{-2i\theta}\right)^j\left(-1\right)^{t-j}}.
\end{equation}
This leads us to
\[A=e^{i\left(n-1\right)\theta}\sum_{j=0}^{k-1}{\left(e^{-2i\theta}\right)^j\sum_{t=j}^{k-1}{\left(\begin{matrix}n-1\\t\\\end{matrix}\right)\frac{\left(2k-2-t\right)!}{\left(k-1-t\right)!}\left(\begin{matrix}t\\j\\\end{matrix}\right)\left(-1\right)^{t-j}}}\]
\begin{equation}
	=\sum_{j=0}^{k-1}{e^{\left(n-2j-1\right)i\theta}\sum_{t=j}^{k-1}{\left(\begin{matrix}n-1\\t\\\end{matrix}\right)\frac{\left(2k-2-t\right)!}{\left(k-1-t\right)!}\left(\begin{matrix}t\\j\\\end{matrix}\right)\left(-1\right)^{t-j}}}
\end{equation}
\begin{equation}
	\Longrightarrow \mathfrak{Im}(A)=\sum_{j=0}^{k-1}{\sin{(\left(n-2j-1\right)\theta)}\sum_{t=j}^{k-1}{\left(\begin{matrix}n-1\\t\\\end{matrix}\right)\frac{\left(2k-2-t\right)!}{\left(k-1-t\right)!}\left(\begin{matrix}t\\j\\\end{matrix}\right)\left(-1\right)^{t-j}}}.
\end{equation}
Using equation \eqref{f_n_short}, we conclude that
\begin{align}
	x[n] &= \frac{(-1)^{k-1}}{(k-1)!}~r^{n-2k} (2i \sin{\theta})^{1-2k} \times \\
	&\qquad \left(2i \sum_{j=0}^{k-1} \sin{((n-2j-1)\theta)} \sum_{t=j}^{k-1} \left(\binom{n-1}{t} \frac{(2k-2-t)!}{(k-1-t)!} \binom{t}{j} (-1)^{t-j}\right)\right) \nonumber \\
	&= \frac{r^{n-2k}}{(k-1)!(\sin{\theta})^{2k-1}2^{2k-2}} \times \\
	&\qquad \sum_{j=0}^{k-1} (-1)^{j} \sin{((n-2j-1)\theta)} \sum_{t=j}^{k-1} \left(\binom{n-1}{t} \frac{(2k-2-t)!}{(k-1-t)!} \binom{t}{j} (-1)^{t}\right).
\end{align}
Note that we had assumed that $n\geq2k$. To combine this with the fact that $x[n]=0$ for $n<2k$ (obtained from lemma \ref{lemma}), we can write
\begin{align} \label{intermediate_f}
	x[n] &= \frac{r^{n-2k}u[n-2k]}{(k-1)!(\sin{\theta})^{2k-1}2^{2k-2}} \nonumber \\
	&\quad\times \sum_{j=0}^{k-1} (-1)^{j} \sin{((n-2j-1)\theta)} \sum_{t=j}^{k-1} \left(\binom{n-1}{t} \frac{(2k-2-t)!}{(k-1-t)!} \binom{t}{j} (-1)^{t}\right).
\end{align}
To simplify the equation \eqref{intermediate_f} more, we state the following lemma
\begin{lemma}
	\begin{equation} \nonumber
		\sum_{t=j}^{k-1} \binom{n-1}{t} \frac{(2k-2-t)!}{(k-1-t)!} \binom{t}{j} (-1)^{t}
	\end{equation}
	\vspace{-0.25cm}
	\begin{equation} \label{internal_summation}
		\begin{aligned}
			&= c_{j,k} (n-1) \ldots (n-j) \times (n-(k+1+j)) \ldots (n-(2k-1)) \\
			&= c_{j,k} (n-1)_j (n-(k+1+j))_{k-1-j}~,
		\end{aligned}
	\end{equation}
	where $c_{j,k}$ is a fixed number (with respect to $n$) and we will find it in the next theorem. Furthermore, the notation $\left(x\right)_n$ is the falling factorial defined in equation \eqref{falling_factorial}.
\end{lemma}

\begin{proof}
	First of all, we know that the result of the series is a $(k-1)$ degree polynomial in terms of $n$. Therefore, it suffices to find $k-1$ roots for it. We also note that
	\begin{equation}
		\left(\begin{matrix}n-1\\t\\\end{matrix}\right)=\frac{\left(n-1\right)\ldots(n-t)}{t!}.
	\end{equation}
	Thus, for each $t\geq j$, we have:
	\begin{equation}
		\left(\begin{matrix}n-1\\t\\\end{matrix}\right)=\left(n-1\right)\ldots\left(n-j\right)f\left(n\right),
	\end{equation}
	where $f(n)$ is a polynomial in terms of $n$. Therefore
	\begin{equation}
		\begin{aligned}
			\sum_{t=j}^{k-1} \binom{n-1}{t} \frac{(2k-2-t)!}{(k-1-t)!} \binom{t}{j} (-1)^t &= (n-1) \ldots (n-j) \sum_{t=j}^{k-1} f_t(n) \\
			&= (n-1) \ldots (n-j) g(n).
		\end{aligned}
	\end{equation}
	where $f_t$ and $g$ are also polynomials of $n$.\newline \newline
	We aim to show that $k+1+j, k+2+j, \ldots, 2k-1$ are some roots of
	\begin{equation}
		\Phi\left(n\right)=\sum_{t=j}^{k-1}{\left(\begin{matrix}n-1\\t\\\end{matrix}\right)\frac{\left(2k-2-t\right)!}{\left(k-1-t\right)!}\left(\begin{matrix}t\\j\\\end{matrix}\right)\left(-1\right)^t};
	\end{equation}
	i.e., if $n=k+\rho$ where $1+j\le\rho\le k-1$, then $\Phi(n)=0$.
	
    By expanding $\Phi(k+\rho)$ and performing algebraic manipulations, we obtain
	\begin{equation}
		\begin{aligned}
			\Phi(k+\rho) &= \sum_{t=j}^{k-1} \binom{k+\rho-1}{t} \frac{(2k-2-t)!}{(k-1-t)!} \binom{t}{j} (-1)^t \\
			&= \sum_{t=j}^{k-1} \frac{(k+\rho-1)!}{t!(k+\rho-1-t)!} \frac{(2k-2-t)!}{(k-1-t)!} \frac{t!}{j!(t-j)!} (-1)^t \\
			&= \frac{(k+\rho-1)!}{j!} \sum_{t=j}^{k-1} \frac{1}{(k+\rho-1-t)!} \frac{(2k-2-t)!}{(k-1-t)!} \frac{1}{(t-j)!} (-1)^t \\
			&= \frac{(k+\rho-1)!}{j!(k-1-j)!} \sum_{t=j}^{k-1} \frac{(2k-2-t)!}{(k+\rho-1-t)!} \frac{(k-1-j)!}{(k-1-t)!(t-j)!} (-1)^t \\
			&= \frac{(k+\rho-1)!}{j!(k-1-j)!} \sum_{t=j}^{k-1} \frac{(2k-2-t)!}{(k+\rho-1-t)!} \binom{k-1-j}{t-j} (-1)^t.
		\end{aligned}
	\end{equation}
	By letting $w=t-j$, we will have
	\begin{equation}
	\Phi\left(k+\rho\right)=\frac{\left(-1\right)^j\left(k+\rho-1\right)!}{j!\left(k-1-j\right)!}\sum_{w=0}^{k-1-j}{\frac{\left(2k-2-j-w\right)!}{\left(k+\rho-1-j-w\right)!}\left(\begin{matrix}k-1-j\\w\\\end{matrix}\right)\left(-1\right)^w}.
	\end{equation}
	Moreover, $\sigma=k-1-j$ implies that
	\begin{equation}
		\Phi\left(k+\rho\right)=\frac{\left(-1\right)^j\left(k+\rho-1\right)!}{j!\left(k-1-j\right)!}\sum_{w=0}^{\sigma}{\frac{\left(2\sigma+j-w\right)!}{\left(\sigma+\rho-w\right)!}\left(\begin{matrix}\sigma\\w\\\end{matrix}\right)\left(-1\right)^w}.
	\end{equation}
	Hence, it suffices to prove that
	\begin{equation}
		\sum_{w=0}^{\sigma}{\frac{\left(2\sigma+j-w\right)!}{\left(\sigma+\rho-w\right)!}\left(\begin{matrix}\sigma\\w\\\end{matrix}\right)\left(-1\right)^w}=0,
	\end{equation}
	where $1+j\le\rho\le\sigma+j$. Note that
	\begin{equation} \label{main_eq_product}
		\frac{\left(2\sigma+j-w\right)!}{\left(\sigma+\rho-w\right)!}=\left(2\sigma+j-w\right)\left(2\sigma+j-w-1\right)\ldots\left(2\sigma+j-w-\nu\right),
	\end{equation}
	where $\nu=\sigma+j-\rho-1$. Therefore, $-1\le\nu\le\sigma-2$. We can rewrite \eqref{main_eq_product} like this:
	\begin{equation}
		\frac{\left(2\sigma+j-w\right)!}{\left(\sigma+\rho-w\right)!}=\left(\sigma-w+\left(\sigma+j\right)\right)\left(\sigma-w+\left(\sigma+j-1\right)\right)\ldots\left(\sigma-w+\left(\sigma+j-\nu\right)\right),
	\end{equation}
	which is equal to
	\[\sum_{p=0}^{\nu+1}{a_p{(\sigma-w)}^p},\]
	where $a_p$'s are some coefficients independent of $w$ and they can be calculated by Vieta's formulas. Consequently, it suffices to show that for each $p$ such that $0\le p\le\sigma-1$,
	\begin{equation}
		\sum_{w=0}^{\sigma}{\sum_{p=0}^{\nu+1}{a_p{(\sigma-w)}^p}\left(\begin{matrix}\sigma\\w\\\end{matrix}\right)\left(-1\right)^w}=0.
	\end{equation}
	By changing the order of summation, and considering that $a_p$'s are independent of $w$, it's equivalent to prove
	\begin{equation}
		\sum_{p=0}^{\nu+1}{a_p\sum_{w=0}^{\sigma}{(\sigma-w)}^p\left(\begin{matrix}\sigma\\w\\\end{matrix}\right)\left(-1\right)^w}=0,
	\end{equation}
	which is true, since
	\[\sum_{w=0}^{\sigma}{{(\sigma-w)}^p\left(\begin{matrix}\sigma\\w\\\end{matrix}\right)\left(-1\right)^w}\]
	represents the number of ways of putting $p$ distinct balls into $\sigma$ distinct boxes such that none of the boxes is empty \cite[Theorem 1.1]{duran-1974}. Given that $0\le p\le\sigma-1$,
	\begin{equation}
		\sum_{w=0}^{\sigma}{{(\sigma-w)}^p\left(\begin{matrix}\sigma\\w\\\end{matrix}\right)\left(-1\right)^w}=0.
	\end{equation}
\end{proof}

\begin{corollary}
	\begin{equation} \label{C}
		c_{j,k}=\left(-1\right)^{k-1}\left(\begin{matrix}k-1\\j\\\end{matrix}\right).
	\end{equation}
\end{corollary}

\begin{proof}
	To prove this, we note that both sides of \eqref{internal_summation} are polynomials of $n$, so their leading coefficients must match. The leading coefficient of the LHS of \eqref{internal_summation}, i.e., the coefficient of $n^{k-1}$, is
	\begin{equation}
		\frac{1}{\left(k-1\right)!}\frac{\left(2k-2-\left(k-1\right)\right)!}{\left(k-1-\left(k-1\right)\right)!}\left(\begin{matrix}k-1\\j\\\end{matrix}\right)\left(-1\right)^{k-1}=\left(-1\right)^{k-1}\left(\begin{matrix}k-1\\j\\\end{matrix}\right).
	\end{equation}
	Hence, The leading coefficient of the RHS of \eqref{internal_summation}, which is $c_{j,k}$, is equal to
	\[\left(-1\right)^{k-1}\left(\begin{matrix}k-1\\j\\\end{matrix}\right).\]
\end{proof}

\begin{corollary} using \cref{intermediate_f,internal_summation,C}, we conclude that
	\[\mathcal{Z}^{-1}\left(\frac{1}{\left(z^2-2az+\left(a^2+b^2\right)\right)^k}\right)\]
	\vspace{-0.5cm}
	\begin{align}  \label{corollary2}
		&=\frac{r^{n-2k}u[n-2k]}{(k-1)!\left(\sin{\theta}\right)^{2k-1}2^{2k-2}} \nonumber \\
		&\quad\times\sum_{j=0}^{k-1}{\left(-1\right)^{j}\sin{\left(\left(n-2j-1\right)\theta\right)}\sum_{t=j}^{k-1}{\left(\begin{matrix}n-1\\t\\\end{matrix}\right)\frac{\left(2k-2-t\right)!}{\left(k-1-t\right)!}\left(\begin{matrix}t\\j\\\end{matrix}\right)\left(-1\right)^{t}}} \nonumber \\
		&=\left(\frac{-1}{4}\right)^{k-1}\frac{r^{n-2k}u\left[n-2k\right]}{(k-1)!\left(\sin{\theta}\right)^{2k-1}} \nonumber \\
		&\quad\times\sum_{j=0}^{k-1}{\left(-1\right)^j\left(\begin{matrix}k-1\\j\\\end{matrix}\right)\left(n-1\right)_j\left(n-\left(k+1+j\right)\right)_{k-1-j}\sin{\left(\left(n-2j-1\right)\theta\right)}}.
	\end{align}
 Note that
 \begin{align}
     &\frac{1}{(k-1)!}\left(\begin{matrix}k-1\\j\\\end{matrix}\right)\left(n-1\right)_j\left(n-\left(k+1+j\right)\right)_{k-1-j}\nonumber\\
     &= \frac{1}{(k-1)!}\times\frac{(k-1)!}{j!(k-1-j)!}\left(j!\binom{n-1}{j}\right)\left((k-1-j)!\binom{n-(k+1+j)}{k-1-j}\right)\nonumber\\
     &=\binom{n-1}{j}\binom{n-(k+1+j)}{k-1-j}
 \end{align}
 Consequently, we can express \eqref{corollary2} as
 \begin{align}
     &\left(\frac{-1}{4}\right)^{k-1}\frac{r^{n-2k}u\left[n-2k\right]}{\left(\sin{\theta}\right)^{2k-1}} \nonumber \\
		&\quad\times\sum_{j=0}^{k-1}{\left(-1\right)^j\left(\begin{matrix}n-1\\j\\\end{matrix}\right)\binom{n-\left(k+1+j\right)}{k-1-j}\sin{\left(\left(n-2j-1\right)\theta\right)}} \nonumber \\
    &=\frac{2~(-1)^{k-1}r^{n-2k}}{\left(2\sin{\theta}\right)^{2k-1}}\sum_{j=0}^{k-1}{\left(-1\right)^j\binom{n-1}{j}\binom{n-(k+1+j)}{k-1-j}\sin{\left(\left(n-2j-1\right)\theta\right)}}.
 \end{align}
 The term $u\left[n-2k\right]$ was removed according to the Remark \eqref{remark_1}.
\end{corollary}

\newpage
\section{Another Proof for the Lemma \ref{lemma}} \label{app:appendixB}
\begin{proof} We utilize the property that $\mathcal{Z}(x[n]*y[n])=X(z)Y(z)$, where $*$ denotes the convolution operator.
	\begin{equation}
		\begin{aligned}
			x[n] &= \mathcal{Z}^{-1}\left(\frac{1}{\left(z^2-2az+\left(a^2+b^2\right)\right)^k}\right) \\
			&= \mathcal{Z}^{-1}\left(\frac{1}{\left(z-(a+bi)\right)^k}\times\frac{1}{\left(z-(a-bi)\right)^k}\right) \\
			&= \binom{n-1}{n-k}(a+ib)^{n-k}u[n-k]\ast\binom{n-1}{n-k}(a-ib)^{n-k}u[n-k] \\
			&= \sum_{m=0}^{\infty}\binom{m-1}{m-k}(a+ib)^{m-k}u[m-k]\binom{n-m-1}{n-m-k}(a-ib)^{n-m-k}u[n-m-k] \\
			&= \sum_{m=k}^{n-k}\binom{m-1}{m-k}(a+ib)^{m-k}\binom{n-m-1}{n-m-k}(a-ib)^{n-m-k}.
		\end{aligned}
	\end{equation}
	Therefore, if $n<2k$, then $n-k<k$, so $x[n]$ will be $0$.
\end{proof}



\end{document}